\documentclass[12pt,reqno]{amsart}
  
\usepackage{latexsym}
\usepackage{amssymb, tikz}

\usepackage{amsthm}
\theoremstyle{plain}

\newtheorem*{theorem*}{Theorem}

\renewcommand{\epsilon}{\varepsilon}

\newcommand{\N}{{\mathbb N}}
\newcommand{\R}{{\mathbb R}}
\newcommand{\C}{{\mathbb C}}

\newcommand{\Z}{{\mathbb Z}}

\newcommand{\E}{{\mathbf E}}
\newcommand{\Prob}{{\mathbf P}}

\renewcommand{\phi}{\varphi}

\numberwithin{equation}{subsection}
\newtheorem{theo}[equation]{{\sc Theorem}}

\newtheorem{cor}[equation]{{\sc Corollary}}

\newtheorem{lem}[equation]{{\sc Lemma}}

\newtheorem{prop}[equation]{{\sc Proposition}}

\theoremstyle{definition}

\theoremstyle{remark}
\newtheorem{rem}[equation]{Remark}

\title[Periods of random eigenfunction clusters]{Expected values of eigenfunction periods}
\author{Suresh Eswarathasan}
\address{Institut des Hautes \'Etudes Scientifiques, 35 route des Chartres, Bures-sur-Yvette, France F-91440} 
\email{suresh@ihes.fr}
\address{Department of Mathematics and Statistics, McGill University, 805 Rue Sherbrooke Ouest, Montr\'eal, Canada}
\email{suresh@math.mcgill.ca}

\date{}

\begin{document}

\begin{abstract}
Let $(M,g)$ be a compact Riemannian surface.  Consider a family of $L^2$ normalized Laplace-Beltrami eigenfunctions, written in the semiclassical form $-h_j^2\Delta_g \phi_{h_j} = \phi_{h_j}$, whose eigenvalues satisfy $h h_j^{-1} \in (1, 1 + hD]$ for $D>0$ a large enough constant.  Let $\Prob_h$ be a uniform probability measure on the $L^2$ unit-sphere $S_h$ of this cluster of eigenfunctions and take $u \in S_h$.  Given a closed curve $\gamma \subset M$, there exists $C_{1}(\gamma, M), C_{2}(\gamma, M) > 0$ and $h_0>0$ such that for all $h \in (0, h_0],$
\begin{equation*}
 C_1 h^{1/2} \leq \E_{h} \bigg[ \big| \int_{\gamma} u \, d \sigma \big| \bigg] \leq C_2 h^{1/2} .
\end{equation*}
This result contrasts the deterministic $\mathcal{O}(1)$ upperbounds obtained by Chen-Sogge \cite{CS}, Reznikov \cite{Rez}, and Zelditch \cite{Zel}. Furthermore, we treat the higher dimensional cases and compute large deviation estimates.  Under a measure zero assumption on the periodic geodesics in $S^*M$, we can consider windows of small width $D=1$ and establish a $\mathcal{O}(h^{1/2})$ estimate.  Lastly, we treat probabilistic $L^q$ restriction bounds along curves.
\end{abstract}

\maketitle

\section{Introduction and Main Results}

\subsection{Introduction}
Let $(M,g)$ be a smooth compact Riemmanian manifold without boundary and $\phi_{h}$ be an $L^2$ normalized eigenfunction of the Laplace-Beltrami operator, written in the semiclassical form $-h^2\Delta_g \phi_{h} = \phi_{h}$. The quantities known as ``periods" $I_{\gamma, h}$, where 
\begin{equation} \label{period}
I_{\gamma, h} = \big| \int_{\gamma} \phi_{h} \, d \sigma \big|
\end{equation}
for $\gamma$ a smooth closed curve on $M$ with arclength measure $d \sigma$, have garnered much interest as of late due to its connection with understanding nodal domains of eigenfunctions \cite{JZ}. Such quantities are a complement to the study of restricted $L^p$ norms \cite{BGT, Hu} and quantum ergodic restriction \cite{TZ, DZ}, and give us information on the fluctuations of eigenfunctions along a curve.

For perspective, let us recall the global $L^p$ estimates of Sogge \cite{Sogg} in dimension $2$: There exists $C>0$ such that $\| \phi_h \|_{L^p(M)} \leq C h^{-\delta(p)}$ where $\delta(p)=\frac{1}{2} - \frac{2}{p}$ for $p \geq 6$ and $\delta(p)= \frac{1}{2}(\frac{1}{2} - \frac{1}{p})$ for $p \leq 6$.  The results for $L^p$ restriction along curves, due in full generality to Burq, G\'erard, and Tzvetkov \cite{BGT}, state: For any finite-length smooth curve $\gamma$, there exists $C>0$ such that $\| \phi_h \|_{L^p(\gamma)} \leq C h^{-\eta(p)}$ where $\eta(p) = \frac{1}{4}$ for $2 \leq p \leq 4$ and $\eta(p) = \frac{1}{2} - \frac{1}{p}$ for $p \geq 4$.

The question of asymptotics of (\ref{period}), which do not seem to follow trivially from $L^2$ restriction bounds even though an eigenfunction period is the zeroth coefficient of the eigenfunction's Fourier series expansion along the curve,  was initially posed on compact hyperbolic surfaces by Good and Hejhal \cite{Go, Hej}.  They showed, by using the Kuznecov trace formula, that
\begin{equation} \label{prevbound}
I_{\gamma, h} \leq C_{\gamma}
\end{equation}
for some $C_{\gamma}>0$ and all $h \in (0,1]$.  This demonstrates a stark constrast to the $L^p$ results of \cite{Sogg} and \cite{BGT}.
Subsequent work by Chen and Sogge \cite{CS} generalized this bound to unit-length geodesics $\gamma$ on compact surfaces using the H\"ormander parametrix, while Reznikov \cite{Rez} proved the bound in the case of arithmetic surfaces using representation theory.  However, before these results were proved, Zelditch \cite{Zel} generalized the Kuznecov trace formula to compact manifolds and obtained (\ref{prevbound}) for any closed hypersurface using the full power of the global symbol calculus for homogeneous Lagrangian distributions.  It is this trace formula that Zelditch established and the consequential asymptotic
\begin{equation} \label{kuz}
\sum_{h_j^{-1} < h^{-1}} \bigg| \int_{\gamma} \phi_{h_j} \, d \sigma \bigg|^2 = C_{\gamma} h^{-1} + \mathcal{O}(1),
\end{equation}
along with an idea of Burq and Lebeau \cite{BL} that play key roles in our main probabilistic result.  As the Weyl law tells us the number of terms in the sum of (\ref{kuz}) is on the order of $h^{-2}$, it follows that ``most" of the terms should be on the order of $h$. 

Using a simple application of Chebyshev's inequality and the asymptotic (\ref{kuz}), Jung and Zelditch were able to obtain the following upper bound:
\begin{theorem*}  \cite{JZ} 
For any compact surface $M$ and closed curve $\gamma \subset M$, and any $g(t)$  such that $\lim_{t \rightarrow \infty} g(t) = \infty$, there exists a subsequence of eigenfunctions $\phi_{h_j}$ of density one such that
\begin{equation} \label{densityone}
\bigg| \int_{\gamma} \phi_{h_j} \, d \sigma \bigg| = \mathcal{O}\big((g(h_j^{-1})h_j)^{1/2} \big).
\end{equation}
\end{theorem*}
We would like to mention that without using the Kuznecov trace formula, we can replace the estimate (\ref{densityone}) with the little-o estimate $o(1)$ just by assuming a generic condition on $\gamma$ and using the ``quantum ergodic restriction" result of Toth and Zelditch \cite{TZ, DZ}.  Chen and Sogge \cite{CS} were also able to prove that the constant in (\ref{prevbound}) can be replaced with $o(1)$ in the strictly negative curvature setting via a refined analysis with the Hadamard parametrix. Both of these estimates can be related to the ``Random wave conjecture" of Berry \cite{Berr}.  Please see our remarks section and the discussion below. 


\subsection{Main Results}

Let $S_{h}$ denote the $L^2$ unit sphere of the space $E_{h} = \{u \in L^2(M) = \break \sum_{h h_j^{-1} \in [1, 1 + hD_M)} z_j \phi_j(x), z_j \in \C \}$ endowed with a uniform probability measure $\Prob_{h}$, defined in Section 2, for some large enough constant $D>0$ that is independent of $h$.  

Also, let ``$f(h) \simeq g(h)$" be the notation that there exists constants $C_1, C_2 >0$ independent of $h$ and an $h_0 >0$ such that for $h \in (0, h_0]$, we have $C_1 g(h) \leq f(h) \leq C_2 g(h)$.

In this note, we prove the following theorem:

\begin{theo} \label{maintheorem}
 Let $(M,g)$ be a compact Riemannian manifold of dimension $n$ and $S$ be a closed submanifold of dimension $d$ with the induced measure $d \sigma$.  Then for a positive integer $p$ and $u \in S_h$, there exists $h_0>0$ such that  
\begin{equation*} 
\E_{h} \big[ \big| \int_S u \, d \sigma \big|^p  \big] \simeq  h^{\frac{dp}{2}} 
\end{equation*}
for $h \in (0,h_0]$, where $\E_h$ is taken with respect to $\Prob_h$.   

Furthermore, for $M_{1,h}$ being the median value of the random variable $| \int_S u |$,
\begin{equation*}
\Prob_{h} \bigg[ \big| | \int_S u \, d \sigma | - M_{1,h} \big| > r \bigg] \leq C \exp \big( - h^{-d}r^2 \big),
\end{equation*}
where $C>0$ is independent of $h$.
\end{theo}


That is, we are able to prove for random linear combinations of eigenfunctions $u$ in the spectral window $[1, 1 + hD)$, (\ref{densityone}) is improved to a $h^{1/2}$ upperbound and lowerbound when $p=d=1$ and $n=2$, and that the random variable concentrates around the median with $h$.  Moreover, we can set $D=\epsilon$ for any fixed $\epsilon>0$ under a measure zero assumption for the set of periodic geodesics in $S^*M$, but must settle for a $h^{1/2}$ upperbound only; please see Section 2.2 for the precise measure 0 assumption and Corollary \ref{averageperiod}.  We also establish a deviation estimate for renormalized random variables when $p>1$.  However, due to the fact that our $p$th moments are just the $p$th powers of our random variables, the most meaningful deviation estimate occurs when $p=1$.  

As a simple justification for this estimate, consider the Gaussian random plane waves on the 2-torus of the form $\frac{1}{k_h} \sum _{j=1}^{N_h} a_j \exp(i h_j^{-1} \cdot x)$ for $h|h_j^{-1}| \in [1, 1 +Dh]$.  Here, $N_{h}$ is equal to the dimension of the eigenspace associated to the window $[1, 1 + Dh]$, $D>0$ is large enough such that $N_h \sim h^{-1}$ by Weyl's Law \cite{Ho}, $k_{h}$ a normalizing factor that is asymptotic to $N_{h}^{1/2}$, and the coefficients $a_j$ are i.i.d. Gaussian random variables.  Simply calculating upperbounds on the period of these quantities demonstrates our asymptotic.   

For the purpose of comparison, let us observe some deterministic cases.  Take $S^2$ and let $Z_{h}$ be a zonal harmonic associated to the eigenvalue $h^{-2}$ and $\gamma$ be a small segment going through the north pole that is also part of a great circle; then $|\int_{\gamma} Z_{h} d \sigma| = \mathcal{O}(h^{1/2})$.  In a large contrast however, letting $\gamma$ be the equator saturates the Chen-Sogge bound (\ref{prevbound}).  Now, on the 2-torus $\mathbb{T}^2$ with eigenfunctions of the form $e^{in x_2}$, taking $\gamma$ to be a small segment of a vertical geodesic gives $|\int_{\gamma}  e^{in x_2} d \sigma| = \mathcal{O}(n^{-1})$.  And similarly to the spherical case, for a given closed geodesic $\gamma$ on $\mathbb{T}^2$, there exist a sequence of eigenfunctions $\phi_j$ such that $\phi_{j|\gamma} = 1$ and therefore saturates the corresponding period bound (\ref{prevbound}) as well.  It is clear that among these two manifolds, which have diametrically opposite spectral theoretic settings, there exists a large range of decay rates.

Our probabilistic results can be interpreted as saying a ``typical" eigenfunction cluster should oscillate enough in order to induce a decaying period on a closed curve, along which the deterministic bounds seemed to be saturated, and that decay better/worse than $h^{1/2}$ is ``atypical".  According to the ``Random wave conjecture" of M. Berry \cite{Berr}, which states the Gaussian random waves model the quantum ergodic case, our estimates appear to be consistent after applying a simple observation using results in \cite{TZ}; please see our remarks section.

We spend a majority of our note proving Theorem \ref{maintheorem}, but also give the following related result on restricted $L^q$ norms along finite-length smooth curves.  The result essentially follows from Theorem 4 in \cite{BL} after making a simple observation.

\begin{theo} \label{subcor}
Given the same setup as Theorem \ref{maintheorem} in the 2-dimensional case with $L_{q, h}$ being the median value of the random variable $\|u\|_{L^q(\gamma)}$ for $q \in [2, \infty)$, we have that for all $h \in (0, 1]$
\begin{equation*}
\Prob_{h} \big( \big| \|u\|_{L^q(\gamma)} - L_{q,h} \big| > r \big) \leq C e^{-c_q h^{-\delta(q)} r^2}
\end{equation*}
where $C, c_q > 0$ and $\delta(q) >0$.  Moreover, given $q \in [2, \infty)$ there exists $h_0$ such that for all $h \in (0,h_0]$, we have that 
\begin{equation*} 
 L_{q,h} \simeq 1.
\end{equation*}
\end{theo}


We would like to mention that \cite{BL} also treats the case of probabilistic $L^{\infty}$ bounds on $M$, from which the restricted $L^{\infty}$ bounds on $\gamma$ follow immediately.  Independently and using different methods, Canzani and Hanin \cite{CH} were able to retrieve these same $L^{\infty}$ upperbounds over aperiodic manifolds $M$ with more precise constants.  

\subsection*{Acknowledgements}
The author would like to thank John Toth for suggesting the problem of eigenfunction periods and numerous helpful discussions.  Furthermore, the author would like to thank Nicolas Burq for the suggestion to look at the paper \cite{BL}, from which this article was inspired, and St\'ephane Nonnenmacher for his valuable comments on previous versions of this paper, as well his useful perspective on quantum ergodic restriction. Finally, the author would like to thank the IH\'ES, where the writing of this article took place.

\section{Probability and measure concentration on spheres} 

The probabilistic setup, and following exposition in the first two sections, we use is patterned after that in \cite{BL}.

\subsection{Measure concentration on spheres} \label{spheremeasures}
In this section, we define the probability measures used in our note. 

Even though we are dealing with $\C^N$, we are only using its properties as a vector space.  Hence, let us take $\R^N$ with the Lebesgue measure $dx = \Pi_{1 \leq j \leq N} d \, x_j$, and consider the uniform probability measure on the unit sphere $S^{N-1} \subset \R^N$ of dimension $N-1$.  This probability measure $\Prob_N$ on $S(N) := S^{N-1}$ is exactly the pushforward of the Gaussian probability measure 
\begin{equation*}
\Pi_{1 \leq j \leq N} \frac{1}{\sqrt{2 \pi}} e^{-|x_j|^2 / 2} d \, x_j
\end{equation*}
on $\R^N$ via the projection map 
\begin{equation*}
\pi : x \rightarrow \frac{x}{|x|},
\end{equation*}
where $|x|$ is the standard Euclidean norm.  The main idea is to consider large $N$ and use these Gaussian probability measures on the corresponding spheres.

For our purposes, we want to compute the corresponding distribution functions explicitly and utilize the concentration of measure phenomenon (to be described more precisely below).  In order to do this, let us fix $M \geq 1$ and group the variables on a large sphere $S(N)$ into $M$ parts.  For $N_j \geq 1$, we consider $N = \sum_{j=1}^M N_j$.  Using this decomposition of $N$, we can rewrite the coordinates on $\R^N$ as $(x_1,...,x_M)$, with each component $x_j$ having the polar decomposition $x_j = \rho_j \omega_j \in \R^{N_j}$ for $\rho_j > 0$ and $\omega_j \in S^{N_j-1}$.  


As we need a particular representation of the distribution function $\Prob_N$, we consider the decomposition of $N = 2 + (N-2)$ where $M=2$.  An explicit calculation shows that 
\begin{equation}\label{distfunc}
\Prob_N(|x_1| > t) = 1_{t \in [0,1)}(1-t^2)^{N/2 - 1},
\end{equation}
where $(x_1, x_2) \in S(N)$ with $N_1=2$ and $N_2= N - 2$.

Furthermore, we have the following deviation formula. 

\begin{prop} 
Suppose $F$ is a Lipschitz function on the sphere $S^{N - 1} = S(N)$, endowed with the uniform probability measure $\Prob_N$ given above, with respect to the natural geodesic distance.  For the median value $M(F)$, we have that for all $r>0$
\begin{equation}
\Prob_N( |F - M(F) | > r) \leq 2 e^{-(N-2)\frac{r^2}{2\|F\|^2_{Lip}}}. \label{measconc}
\end{equation}
This large deviation estimate is more commonly referred to as the ``concentration of measure phenomenon".
\end{prop}
The factors in the exponential essentially say that the width of the distribution for $F$ is on the order of $\frac{\|F\|_{Lip}}{\sqrt{N}}$ for $N$ large.  Please see \cite{Led} for further results on measure concentration.

\subsection{Probabilistic decomposition of $L^2(M)$} 

Although the results we present were inspired by the eigenfunction treatment in \cite{BL}, these results were mainly intended for probabilistic applications towards the damped and non-linear wave equations.  The probabilistic treatment in \cite{BL} uses the dyadic Littlewood-Paley decomposition in harmonic analysis through eigenfunction clusters. For more on the relationship between Littlewood-Paley theory and probabilistic methods, please see \cite{BL}.

\subsubsection*{General manifolds}
Take two $h$ dependent sequences $a_h$ and $b_h$ (with limits $a$ and $b$, respectively) with the relation that $a_{h} < b_{h}$ for all $h \in (0, 1]$ and
\begin{equation*}
0 \leq \lim_{h \rightarrow 0} a_{h} \leq \lim_{h} b_{h}.
\end{equation*}
Furthermore, in the case that $a = \lim a_{h} = b= \lim b_{h}$ exist, we impose $a>0$ and that the rate of decay cannot be too fast, i.e.
\begin{equation*}
D h \leq b_{h} - a_{h}
\end{equation*}
for some large constant $D>0$ that we will specified shortly. 

As $(M,g)$ is a compact manifold with the semiclassical relation $-h_j^2\Delta_{g} \phi_j = \phi_j$ and $L^2$ orthonormal basis of eigenfunctions $\{ \phi_{j} \}_{j=1}^{\infty}$, we can consider the set
\begin{equation*}
I_{h} := \{ k \in \N \, : \, h h_k^{-1} \in (a_{h}, b_{h}] \}
\end{equation*}
and the spectral cluster
\begin{equation*}
E_{h} = \{ u = \sum_{k \in I_{h}} z_k \phi_k(x), \, z_k \in \C \},
\end{equation*}
which is independent of the choice of eigenbasis.  Lastly, set $N_{h} = dim_{\C} E_{h}$.  The Weyl formula, in its semiclassical form (please see \cite{Zw}) where $\lambda^{-1} = h$, states that there exists a constant $C>0$ independent of $h$ such that
\begin{equation} \label{weyl}
\big| N_h - c_n \frac{Vol(M)}{(2\pi)^n} \big( (h^{-1}b_h)^n - (h^{-1} a_h)^n \big)  \big| \leq C h^{-(n-1)}.
\end{equation}
A further calculation shows that $N_h$ is asymptotic to either $c_n \frac{Vol(M)}{(2\pi)^n}h^{-n}(b_h^n - a_h^n)$ when $0 \leq a < b$ or $c_n n a^{n-1} \frac{Vol(M)}{(2\pi)^n}h^{-n} \big( (b_h - a_h) + \mathcal{O}(h + (b_h - a_h)^2 + (b_h - a_h)|a-a_h|) \big)$ when $0< a = b$.

In either case, there exists $D_M > 0$ and constants $0< \alpha < \beta$ such that for $D_M h \leq b_h - a_h$ and $h \in (0, 1]$,  
\begin{equation} \label{remterm}
\alpha h^{-n} (b_h - a_h) \leq N_h \leq \beta h^{-n} (b_h - a_h)
\end{equation}
which implies $N_h \geq 2$.  We choose the constant $D_M>0$ in this particular way in order to guarantee that $N_h$ properly goes to infinity with $h$, which is important for our asymptotic results.  Otherwise, we may be in a geometrical setting where this is not the case.  For instance, choosing $D_M=\frac{1}{2}$ in the case of the $S^2$ with the round metric would give us $N_h = 0$ for infinitely many $h$.

Although we can consider windows of growing size, we will be primarily concerned with the case $b_h = 1 + D_Mh$ and $a_h = 1$, with the spaces $E_h \subset L^2(M)$ corresponding to the spectral window $(1, 1+ hD_M]$.  Moreover, it follows that the quantity $N_h \simeq h^{-(n-1)}$, suggesting to us that it should be thought of as the remainder term in the semiclassical form of the Weyl law.

\subsubsection*{Aperiodic manifolds}

As the case of small-length spectral windows $(1, 1 + hD]$, for $D>0$ arbitrarily small but fixed, is also of interest, we must make a further dynamical assumption on $M$ in order to obtain (\ref{remterm}).  First, we will review some important results.  For the boundaryless situation, Duistermaat and Guillemin \cite{DG} proved that if the set of periodic geodesics forms a set of measure zero in $S^*M$, then for any $\epsilon > 0$
\begin{equation*}
\big| \#\{k : \, h \lambda_k \in (a_{h} , b_h] \} - c_n \frac{Vol(M)}{(2\pi)^n} \big( (h^{-1}b_h)^n - (h^{-1} a_h)^n \big)  \big| \leq \epsilon h^{-(n-1)},
\end{equation*}
for all $h \leq h_0(\epsilon)$.  This is equivalent to saying that the remainder term $R(\lambda)$ in the Weyl law is $o(\lambda^{n-1})$.  Compact manifolds with this assumption will be referred to as being ``aperiodic".  In the case of manifolds with boundary, Weyl conjectured that there exist two-term asymptotics with $R(\lambda) = o(\lambda^{n-1})$; this was proved in the 1980s by Ivrii \cite{Ivr} under the same measure zero assumption for the billiard flow (this result was also proved independently by Melrose \cite{Mel} using some stronger assumptions involving convexity) .

Hence, choosing $\epsilon_0 > 0$ small enough implies there exists $\beta > \alpha >0$ such that
\begin{equation} \label{remterm2}
\alpha h^{-(n-1)} \leq N_h \leq \beta h^{-(n-1)}
\end{equation}
for $h \in [0, h_0(\epsilon_0) )$ in the case of small-length $D$ spectral windows.

\subsubsection*{Probability measure}
We are able to endow the unit sphere $S_h \subset E_h$, with respect to the $L^2$ norm, with the uniform probability measure $\Prob_{N_h}$ defined in the previous section.  An important set of random variables we will consider, for $u \in S_h$, are defined in the following way:
\begin{equation} \label{specclust}
ev_x(u) = u(x) = \sum_{h \lambda_k \in (a_h, b_h]} z_k \phi_k(x) = \langle z, \frac{b_{x,h}}{|b_{x,h}|} \rangle  |b_{x,h}| 
\end{equation}
for $b_{x,h} = (\phi_k(x))_{k=1}^{N_h}$.  Clearly, the above representation only holds for $x$ such that $b_{x,h} \neq 0$, and is the only situation of relevance to our results as we will be computing probabilites of positive quantities.    Notice that $|b_{x,h}|^2 = \frac{N_h}{Vol(M)} + \mathcal{O}(h^{-(n-1)})$ by the pointwise Weyl law \cite{Ho}.

For the simplicity of notation, we will abbreviate $\Prob_{N_h}$ by using the symbol $\Prob_h$ in the later sections.

\section{Generalized Kuznecov trace formulas}

The tool of trace formulas has a rich history and has deeply influenced the spectral theory of automorphic forms, as well as spectral geometry.  Even to today, microlocal analysts are finding new ways to use trace formulas to derive spectral data on manifolds via dynamical information.  For an extensive survey on the far-reaching effects of semiclassical trace formulae and their applications towards the spectral theory of the Laplace-Beltrami operator, please see the survey \cite{CdV} of Colin de Verdiere.

For our purposes, we will be concentrating only on the Kuznecov trace formula \cite{Ku}, whose original formulation was  
\begin{equation*}
\int_{Y_1} \int_{Y_2} U(t,x,y) \, d \mu_1(x) \, d \mu_2(y) = \sum_{[\sigma] \in \Gamma_{Y_1} \backslash (\Gamma / \Gamma_{Y_2})} I_t( [\sigma] ),
\end{equation*}
where $U(t,x,y)$ is the full half-wave kernel $e^{-it \sqrt{\Delta}}$ on a compact quotient $M=\Gamma / D$ of a non-compact symmetric space $D$, $Y_i$ are closed geodesics for $i=1,2$ on $M$, $\Gamma_{Y_i}$ is the isotropy group in $\Gamma$ of $Y_{i}$, and $I_t([\sigma])$ is a distribution on $\R$ that is invariantly associated to $[\sigma]$.  This notation was generalized to compact Riemannian manifolds by Zelditch \cite{Zel}.


We now provide the spectral asymptotics that were computed in \cite{Zel} from the Kuznevoc trace formula.  The following is a key estimate in the proof of Theorem \ref{maintheorem}.

\begin{prop}
For a closed submanifold $S$ of dimension $d$ with surface measure $d \sigma$, we have
\begin{equation*}
\bigg| \sum_{\{j: h_j^{-1} \leq h^{-1}\}} \big| \int_S \phi_j(s) \, d \sigma(s) \big|^2 - c_n Vol(S(N^*S))h^{-(n-d)} \bigg| \leq c_S h^{-(n-d)+1}
\end{equation*}
for $N^*S$ denoting the normal bundle over $S$ and $S(N^*S)$ denoting its unit sphere bundle.
\end{prop}

Now, let us set
\begin{equation*}
E(h, S) := \sum_{j: h_j^{-1} \leq h^{-1}} | \int_S \phi_j(s) d \sigma(s)|^2,
\end{equation*}
which is also independent of the choice of eigenbasis.  We can consider an analogous quantity to $N_h$, which we will label as $N(S)_h$ and set it as
\begin{equation*}
N(S)_h := E(h^{-1}b_h, S) - E(h^{-1}a_h, S).
\end{equation*}
Moreover, 
\begin{equation} \label{newweyl}
\big| N(S)_h - c_nVol(S(N^*S)) \big( (h^{-1}b_h)^{n-d} - (h^{-1}a_h)^{n-d} \big) \big| \leq C_S h^{-((n-d)-1)}
\end{equation}
for $h \in (0, 1]$. Furthermore, we set $b_{S,h} :=(\int_S \phi_j(s) d \sigma(s))_{j \in I_h} \in \C^{N_h}$.  Similarly to the pointwise Weyl law after (\ref{specclust}), $N(S)_h = |b_{S,h}|^2$. 

Following the calculations after (\ref{weyl}), we see that there exists $\tilde{D}_S > 0$ and constants $0< \alpha < \beta$ such that for $\tilde{D}_S h \leq b_h - a_h$ and $h \in (0, 1]$,  
\begin{equation} \label{newremterm}
\alpha h^{-(n-d)} (b_h - a_h) \leq N(S)_h \leq \beta h^{-(n-d)} (b_h - a_h).
\end{equation}


\section{Proof of main theorem}

The main idea in the article \cite{BL} is that by having spectral asymptotics associated to a summation formula (in our case, the eigenvalue counting function and the Kuznecov sum formula), we can obtain explicit expected values and deviation estimates.  Therefore, let us consider the random variable defined by
\begin{equation} \label{innerprod2}
\int_S u(s) d \sigma(s) = \int_S \sum_{k \in I_h} z_ke_k(s) d \sigma(s) = \langle z, b_{S,h} \rangle = \langle z, \frac{b_{S,h}}{|b_{S,h}|} \rangle |b_{S,h}|,
\end{equation}
where $d \sigma$ is the surface measure on $S$. Once again, we note that the above representation only holds for $|b_{S,h}| \neq 0$, which is exactly our case since we will compute $\Prob_h[| \int_{S} u| > \lambda]$ for $\lambda \geq 0$. 

We note again that it is useful to view $N(S)_h=|b_{S,h}|^2$ as sort of remainder term for the asymptotics of $E(h^{-1}b_h, S)$.  For instance, if we again consider the spectral window of large size $D$, i.e. $b_h = 1 + hD$ and $a_h = 1$ where $D$ is the supremum of the constants in (\ref{remterm}) and (\ref{newremterm}), then it is clear  that $N(S)_h \simeq h^{-((n-d)-1)}$.  Notice that for $n=2$ and $d=1$, $N(S)_h \simeq 1$.

\begin{lem} \label{Lip1}
For an integral $p \geq 1$ and $u \in S_h$, there exists $C_p>0$ such that for $h \in (0, 1]$, we have 
\begin{equation*}
\| F_p(u) \|_{Lip} \leq C_p (N(S)_h)^{p/2}
\end{equation*}
where $F_p(u) = | \int_S u(s) d \sigma(s) |^p = F_1(u)^p$.
\end{lem}

\begin{proof}

We have, after (\ref{innerprod2}) with $u = \sum_{k\in I_h} z_k e_k$ and $v = \sum_{k\in I_h} w_k e_k$, that
\begin{equation*}
\bigg| \big|\int_S u \big|^p - \big| \int_S v \big|^p \bigg| = \bigg| \big| \langle z, \frac{b_{S,h}}{|b_{S,h}|} \rangle |b_{S,h}| \big|^p  - \big| \langle w, \frac{b_{S,h}}{|b_{S,h}|} \rangle |b_{S,h}| \big|^p  \bigg|.
\end{equation*}
Using the fact that $u,v \in S_h$ gives us
\begin{align*}
|b_{S,h}|^p \cdot \bigg| \big| \langle z, \frac{b_{S,h}}{|b_{S,h}|} \rangle \big|  - \big| \langle w, \frac{b_{S,h}}{|b_{S,h}|} \rangle \big|  \bigg| \cdot & \bigg( \sum_{j=0}^{p-1} \big| \langle z, \frac{b_{S,h}}{|b_{S,h}|} \rangle \big|^{p-1-j} \cdot \big| \langle w, \frac{b_{S,h}}{|b_{S,h}|} \rangle \big|^j \bigg), \\
&  \leq p |b_{S,h}|^p \cdot \big| \langle z-w, \frac{b_{S,h}}{|b_{S,h}|} \rangle  \big|. \\
\end{align*}
Also $|b_{S,h}|^p = N(S)_h^{p/2}$, and applying the Cauchy-Schwarz inequality once more shows
\begin{equation*}
\bigg| \big|\int_S u \big|^p - \big| \int_S v \big|^p \bigg| \leq p \cdot N(S)_h^{p/2} \| u-v \|_{L^2(M)}.
\end{equation*}
Finally, as $\| u-v \|_{L^2(M)} = |z-w| \leq C \cdot dist_{S^{N_h - 1}}(z,w) = C \cdot dist_{S(N_h)}(u,v)$, our proposition is proved.
\end{proof}

\begin{theo}(Average value $A_{p,h}$) \label{averval}
\begin{equation*}
A_{p,h} = \E_h \big[ | \int_S u |^p \big] = p \cdot N(S)_h^{p/2} \beta(\frac{p}{2}, N_h),
\end{equation*}
which as $h \rightarrow 0$ gives us the asymptotic
\begin{equation*}
A_{p,h} = C_p \cdot (N(S)_h /N_h)^{p/2} \cdot (1 + o(1))
\end{equation*}
for some $C_p > 0$.
\end{theo}

\begin{proof}
It is known that
\begin{equation*}
\E_h\bigg[ | \int_S u |^p \bigg] = \int_0^{\infty} p \lambda^{p-1} \Prob_h \bigg[ \big| \int_S u \big| > \lambda \bigg] \, d\lambda.
\end{equation*}
Furthermore, after using (\ref{innerprod2}) we obtain 
\begin{equation*}
\Prob_h \bigg[ \big| \int_S u \big| > \lambda \bigg]  = \Prob_h \bigg[ \big|  \langle z, \frac{b_{S,h}}{|b_{S,h}|} \rangle  \big| > \frac{\lambda}{|b_{S,h}|} \bigg].
\end{equation*}
Therefore, after using (\ref{distfunc}), we see that
\begin{align*}
\E_h \bigg[ | \int_S u |^p \bigg] &=p \int_0^{\infty} \lambda^{p-1} \cdot \mathbf{1}_{\frac{\lambda}{|b_{S,h}|} \in [0,1]} \cdot (1 - \frac{\lambda^2}{|b_{S,h}|^2})^{N_h - 1} \, d \lambda \\
&= p \cdot |b_{S,h}|^p \int_0^{1} \eta^{p-1} \cdot (1 - \eta^2)^{N_h - 1} \, d \lambda \\
&= p \cdot |b_{S,h}|^p \cdot \beta(\frac{p}{2}, N_h).
\end{align*}
Recalling that the beta function has the closed form $\beta(\frac{p}{2}, N_h) = \frac{\Gamma(p/2) \Gamma(N_h)}{\Gamma((p/2) + N_h)}$, and as $N_h \rightarrow \infty$ we have the following well-known asymptotic for fixed $p$:
\begin{equation*}
\beta(\frac{p}{2}, N_h) = \Gamma(p/2) N_h^{-p/2} (1 + o_h(1)).
\end{equation*}
Using $N(S)_h=|b_{S,h}|^2$ completes our proof.
\end{proof}

\begin{theo} (Estimation of the median $M_{p,h}$ and large deviation) \label{estmed}

For all $p \in \Z^+$ and $h \in [0, 1)$,
\begin{equation} \label{medbound}
0 \leq M_{p,h} \leq 2^{p/2}(A_{2,h})^{p/2}
\end{equation}
and the following deviation estimates hold in the range $p < \frac{n}{n-d}$ and $h \in (0, h_0]$:
\begin{equation} \label{medconv}
\bigg|  A_{p,h}- M_{p,h}  \bigg| \leq \frac{\pi}{2} \cdot \big(\frac{2N(S)_h^{p}}{N_h-2} \big)^{\frac{1}{2}}
\end{equation}
and
\begin{equation} \label{dev}
\Prob_h \big( | F_p(u) - M_{p,h} | > r \big) \leq 2 \exp \big(-\frac{N_h-2}{p(N(S)_h)^p} r^2 \big).
\end{equation}
\end{theo}

\begin{proof}
Equation (\ref{medbound}) follows immediately from properties of the median and Chebyshev's inequality.  Then, equation (\ref{medconv}) is just another direct calculation after applying (\ref{measconc}).  Observe that 
\begin{align}
\nonumber \bigg|  A_{p,h} - M_{p,h}  \bigg| &= \bigg|  \|F_p(u)\|_{L^1(S(N_h)} - \|M_{p,h}\|_{L^1(S(N_h))}  \bigg| \\
\nonumber & \leq  \big\| F_p(u) - M_{p,h} \big\|_{L^1(S(N_h))}  \\
\nonumber & =  \int_0^{\infty}  \Prob_h \big[ \big| F_p(u) - M_{p,h} \big| > \lambda \big] \, d \lambda \\
\label{almostgam} & \leq  2 \int_0^{\infty} e^{-\frac{N_h-2}{N(S)_h^p} \lambda^2} \, d \lambda.
\end{align}
Hence, we obtain that (\ref{almostgam}) is equal to
\begin{equation*}
\frac{\pi}{2} \big( \frac{2 N(S)_h^p}{N_h - 2}\big)^{\frac{1}{2}} .
\end{equation*}

For the purpose of obtaining a nontrivial deviation estimate, we must have that $\lim_{h \rightarrow 0} \break \frac{2 N(S)_h^p}{N_h - 2} = 0$, which is only possible if $p < \frac{n}{n-d}$ after using (\ref{remterm}) and (\ref{newremterm}).  Lastly, (\ref{dev}) comes from the Lipschitz estimate on $F_p(u)$ and (\ref{measconc}).
\end{proof}

\begin{cor} \label{renomdev}
Let $\tilde{M}_{p,h}$ be the median value for the renormalized random variable $\tilde{F}_p(u) := \frac{F_p(u)}{pN(S)_h^{p/2}}$.  Then for all $p \in \Z^+$, there exists $C>0$ such that for all $h \in [0, 1)$,
\begin{equation} 
\Prob_h \big( | \tilde{F}_p(u) - \tilde{M}_{p,h} | > r \big) \leq C \exp \big(-N_h r^2 \big).
\end{equation}
\end{cor}

\begin{rem}
Notice the deviation estimate (\ref{dev}) is meaningful only in the case $p=1$, as the deviation is larger than the average for $p > 1$ because $F_p(u) = F_1(u)^p$.  This is our reasoning for renormalizing in Corollary \ref{renomdev}.
\end{rem}

Using the asymptotics in (\ref{remterm}) and (\ref{newremterm}), along with Theorems \ref{averval} and \ref{estmed}, immediately gives us
\begin{cor}
For a spectral window of large enough constant size $D$, 
\begin{equation*}
A_{p,h} \simeq h^{dp/2} . \label{averageperiod}
\end{equation*}
Furthermore, the random variable $F_1(u)$ concentrates around the median $M_{1,h}$.
\end{cor}

\begin{rem}
Under the aperiodicty assumption on $M$, we can consider windows of small constant length $D > 0$, but the asymptotic in Corollary \ref{averageperiod} becomes only a $\mathcal{O}(h^{pd/2})$ bound as we are unable to refine the estimate (\ref{newweyl}) further at this time.  However, the deviation estimate continues to hold.
\end{rem}

\begin{rem}
Based on the analysis made in the proofs of Theorems \ref{averval} and \ref{estmed}, we see that obtaining actual asymptotics with error terms for the eigenfunction periods comes down to deducing more precise asymptotics for the remainders of both the Weyl law and Kuznecov sum formula.  Given the various works that improve on the remainder estimate under different assumptions (for instance \cite{DG, Ivr} and the references therein), it is clear that this problem is difficult.
\end{rem}

\section{Proof of restricted versions of the Burq-Lebeau statistics}
We now present probabilistic restricted $L^q$ estimates in the spectral window $[1, 1+ Dh)$ with $D>0$ large in the general case and with $D>0$ small in the aperiodic case, for which the even periods $ \int_{\gamma} u^{2l} \, d \sigma $ for $l \in \Z^+$ follow immediately.  Before beginning, we want to make note that the proofs for the global $L^q$ bounds in Section 2.3 of \cite{BL} go through with minor changes, such as establishing the new Lipschitz estimate and the new deviation estimate.  

For the sake of simplicity, we write our theorems in the case of finite-length smooth curves in surfaces $\gamma \subset M$.  Using the general theorems in Burq, G\'erard, and Tzvetkov \cite{BGT} that give the corresponding Lipschitz estimates for lower dimensional submanifolds, the higher dimensional formulations of our theorems follow similarly.

\begin{prop} \cite{BGT} 
For the spectral window $[1, 1+ Dh)$ with $D>0$ large enough in the general case, and $[1, 1+ h)$ in the aperiodic case, there exists $C_{q}>0$ such that for $q \in [2, \infty)$ and $u \in S_h$, we have
\begin{equation*}
\big( \int_{\gamma} |u|^{q} \, d \sigma \big)^{1/q} = \| u \|_{L^q(\gamma)}  \leq C_{q} h^{-\delta(q)} \cdot \|u\|_{L^2(M)},
\end{equation*}
where
\begin{equation*}
\delta(q) = \begin{cases} \frac{1}{2} - \frac{1}{q}, & \mbox{ for } q \geq 4 \\ \frac{1}{4}, & \mbox{ for }  2 \leq q \leq 4. \end{cases} 
\end{equation*}
Moreover, the function $F_q(u) = \| u \|_{L^q(\gamma)}$ has Lipschitz norm $C_{q} \cdot h^{-\delta(q)} $.
\end{prop}

The proof of this proposition is similar to that of Lemma \ref{Lip1}. Now, using (\ref{measconc}) and the above estimate results immediately in

\begin{theo} \label{newlip}
Consider $u \in S_h$.   For any $q \in [2, \infty)$, there exists $c_{1,q} > 0$ such that for $h \in (0, h_0]$ with $L_{q,h}$ being the median value of $F_q(u) = \| u \|_{L^q(\gamma)}$,
\begin{equation*}
\Prob_h \big( \big| \|u\|_q - L_{q,h} \big| > r \big) \leq 2 e^{-c_{1,q} G(h) r^2},
\end{equation*}
where $G(h) = h^{-2/q}$ for $q \geq 4$ and $G(h) = h^{-1/2}$ for $2 \leq q \leq 4$.
\end{theo}

\begin{theo}
For $B_{q,h} = \E_h[\| u \|^q_q]^{1/q}$ and $q \in [2, \infty)$, there exists $h_0>0$ such that for $h \in (0, h_0]$ and $C(\gamma,q,M) = \frac{\Gamma(q/2)^{1/q} (Vol(\gamma))^{1/q}}{2^{1/q}\sqrt{2 e Vol(M)}} >0$ we have
\begin{equation} \label{lqavg}
B_{q,h} \simeq C.
\end{equation}
It follows that for a given $q \in [2, \infty)$, there exists $c_1(q, \gamma, M)>0$ such that for all $h \in (0, 1]$,
\begin{equation} \label{meanupper}
 L_{q,h} \leq c_{1}.
\end{equation}  
And, there exists $c_2(q, \gamma, M)>0$ and $h_0>0$ such that for all $h \in (0, h_0]$,
\begin{equation} \label{meanlower}
c_{2} \leq L_{q,h}.
\end{equation}
\end{theo}

\begin{rem}
The constants $c_1(q, \gamma, M)$ and $c_2(q, \gamma, M)$ can be computed in terms of $C(q,\gamma, M)$; please see the following proof.
\end{rem}

\begin{proof}
We re-do the proof of Theorem 4 in \cite{BL}, with the necessary modifications, for the facility of the reader.  Notice that 
\begin{equation*}
\E_h\big[ |g|^q \big] = q \int_{0}^{\infty} \lambda^{q-1} \Prob_h(|g| > \lambda) \, d \lambda,
\end{equation*}
and we obtain, after applying Fubini and (\ref{distfunc}), that
\begin{align*}
\E_h( \| u \|_q^q) & = \int_{\gamma} \int_{S_h} |u(s)|^q \, d \sigma(s) \, d \Prob_h = q\int_{\gamma} \int_0^{\infty} \lambda^{q-1} \Prob_h(|u(s)| > \lambda) \, d \lambda \, d\sigma(s) \\
& = q \int_{\gamma} \int_0^{|b_{s,h}|} \lambda^{q-1} (1-\frac{\lambda^2}{|b_{s,h}|^2})^{N_h-1} \, d \lambda \, d \sigma(s) \\
& = q\big( \int_{\gamma} |b_{s,h}|^{q} \, d \sigma(s) \big) \int_0^1 z^{q-1}(1-z^2)^{N_h-1} \, dz = (B_{q,h})^q. \\
\end{align*}
 We note that 
\begin{equation*}
 \int_0^1 z^{q-1}(1-z^2)^{N_h-1} \, dz = (1/2) \beta(q/2, N_h) = \frac{1}{2} \cdot \Gamma(q/2) \cdot N_h^{-q/2}(1 + o(1)) 
\end{equation*}
for $h \in (0, h_0]$ when $q$ is fixed using basic beta function asymptotics, while 
\begin{equation*}
\int_{\gamma} |b_{s,h}|^{q} \, d \sigma(s)  \simeq \frac{Vol(\gamma)}{Vol(M)^{q/2}}N_h^{q/2} \text{ for } h \in (0, h_0]
\end{equation*}
as $|b_{x,h}|^2 = \frac{N_h}{Vol(M)} + \mathcal{O}(h^{-(n-1)})$ for all $x \in M$ by the pointwise Weyl law on $M$ (see \cite{Ho} for the case of length $D$ windows and \cite{DG, Ivr} for the case of small-length windows with $o(h^{-(n-1)})$ remainder estimate) and proceeding similarly as to obtaining (\ref{weyl}).  Using Stirling's formula for the gamma function $\Gamma(t)$ and the pointwise Weyl law once more, it follows that
\begin{equation} \label{gammaupper}
B_{q,h}^q \leq q C_1^q \big( \frac{N_h}{N_h + (q/2)} \big)^{N_h + (q/2) - (1/2)} \Gamma(q/2)
\end{equation}
for $C_1>0$ and all $h \in [0,1)$.  Hence $B_{q, h} \leq C_{2,q}$ for any $q \in [2, \infty)$ and $h \in [0,1)$, where $C_{2,q}>0$.  Furthermore, we have that 
\begin{equation} \label{almostmean}
B_{q,h} \simeq \frac{\Gamma(q/2)^{1/q}}{2^{1/q}\sqrt{2 e Vol(M)}} (Vol(\gamma))^{1/q}
\end{equation}
as $h \rightarrow 0$ and (\ref{lqavg}) follows.

The inequality (\ref{meanupper}) now follows from Chebyshev's inequality in the following way.  Clearly,
\begin{equation*}
\Prob_h[\|u\|_q > t] \leq \frac{1}{t^q} \E_h[\| u \|_q^q] = (B_{q,h}/t)^q.
\end{equation*}
Setting $t = L_{q,h}$ and using properties of the median, we obtain $L_{q,h} \leq 2^{1/q} B_{q,h}$ and (\ref{meanupper}) follows. We will use (\ref{almostmean}) to prove our last inequality (\ref{meanlower}).

Observe that
\begin{equation*}
| B_{q,h} - L_{q,h}|^q = | \, \|F_q \|_{L^q(S(N_h))} - \|L_{q,h} \|_{L^q(S(N_h))} \, |^q \leq \| F_q - L_{q,h} \|_{L^q(S(N_h))}^q.
\end{equation*}
In our case, as opposed to that in \cite{BL}, our Lipschitz estimate for $F_q$ is $C_{q} \cdot h^{-\delta(q)}$.  Hence, after using Corollary \ref{newlip}, we obtain
\begin{align*}
\| F_q - L_{q,h} \|_{L^q(S(N_h))}^q = & q \int_0^{\infty} \lambda^{q-1} \Prob_h(| F_q - L_{q,h} | > \lambda) \, d \lambda \\
& \leq 2q \int_0^{\infty} \lambda^{q-1} \exp(-c_1 h^{-\delta(q)} \lambda^2) \, d \lambda = \frac{2q}{c_1^{q/2}} h^{q \delta(q)/2} \Gamma(q/2).
\end{align*}
Taking $q$th roots shows the right-hand side is bounded above by $C_{3,q} h^{\delta(q)/2}$ for $C_{3,q}>0$, which converges to $0$ as $h \rightarrow 0$ for $q$ fixed.  This convergence along with the estimate (\ref{almostmean}) concludes the proof of (\ref{meanlower}).

\end{proof}


\section{Remarks}

\begin{rem}
We will now provide a simple argument that gives a $o(1)$ estimate for (\ref{densityone}) without using the Kuznecov trace formula. Let us recall the ``QER" result of \cite{TZ}.  

\begin{theo}
Let $(M, g)$ be a compact surface with ergodic geodesic flow, and $\gamma \subset M$ be a closed curve which is microlocally asymmetric with respect to the geodesic flow, and $\{\phi_j\}_{j \in \mathbb{N}}$ be eigenfunctions of the Laplace-Beltrami operator.  Then there exists a density-one subset S of $\mathbb{N}$ such that for $a \in S^{0,0}(T^*\gamma \times [0, h_0))$,
\begin{equation*}
\lim_{j \rightarrow \infty \; j \in S} \langle Op_{h_j}(a) \phi_{j},  \phi_{j} \rangle_{L^2(\gamma)} = \omega(a)
\end{equation*}
where 
\begin{equation*}
\omega(a) = \frac{4}{vol(S^*\gamma} \int_{B^* \gamma} a_0(s, \sigma)(1-|\sigma|^2)^{-1/2} \, ds d\sigma.
\end{equation*}
\end{theo}
\noindent For the definition of microlocal asymmetry, please see \cite{TZ}.

Now, let us assume $\gamma$ satisfies the microlocal asymmetry assumption with $\{ \phi_{j}\}_{j=1}^{\infty}$ being a quantum ergodic sequence of eigenfunctions on $M$.  Take $\epsilon > 0$.  Notice that $\int_{\gamma} \phi_{j}  d \sigma = \langle 1, \phi_{j} \rangle_{L^2(\gamma)}$.  Set $(s,\tau)$ to be the coordinates on $B^* \gamma$.  Since $WF_h(1) \subset \gamma \times \{0\}$, we have that for a smooth cutoff $\chi_{\epsilon}$ where $\chi_{\epsilon}(\tau) = 1$ on $(-\epsilon, \epsilon)$ and $\chi_{\epsilon}(\tau) = 0$ outside $(-(3/2)\epsilon, (3/2)\epsilon)$,
\begin{equation} \label{qerint}
\langle 1, \phi_{j} \rangle_{L^2(\gamma)} = \langle Op_{h_j}(\chi_{\epsilon})1, \phi_{j} \rangle_{L^2(\gamma)} + \mathcal{O}(h^{\infty}),
\end{equation}
by semiclassical wavefront set calculus (see \cite{Zw} for more details) for the corresponding $h_j$-pseudodifferential cutoff $Op_{h_j}(\chi_{\epsilon})$.  Applying the QER result to the quantity $\langle 1, Op_{h_j}(\chi_{\epsilon})^*\phi_{j} \rangle_{L^2(\gamma)}$ after using the Cauchy-Schwarz inequality shows (\ref{qerint}) is $\mathcal{O}(\epsilon)$ for all $h_j \leq h_0(\epsilon)$, which proves the estimate.

Heurisitically speaking, if we were to express $\phi_j$ as an ``ergodic" one-dimensional Fourier series on the curve, then the above calculation suggests that the mass of the ``period" arises from the constant mode of the series.  It is interesting to ask how one can formalize this notion and apply it to periods of odd powers of ergodic eigenfunctions.  A similar idea is used in \cite{Rez}.

As we can see, using quantum ergodicity assumptions along with asymmetry allows us to retrieve the $o(1)$ bound (under negative curvature assumptions) of Chen-Sogge \cite{CS}  without the use of any trace formula.  From this point of view, our probabilistic results seem to be consistent with the ``Random wave model" of Berry.
\end{rem}

\begin{rem} As mentioned above, the analogous question for $|\int_{S} u^{2l+1} d \sigma|$ for $l \in \Z^+$ is still of interest but is more subtle.  It is not entirely clear that the same methods using Weyl-type formulas will work.  Furthermore, it is not known if we have analogues of (\ref{prevbound}) in the case of all odd powers.  We hope to address these questions in future work.

\end{rem}

\end{document}